\documentclass[12pt,twoside]{amsart}
\usepackage{amsthm, amsfonts, amssymb, amsmath, latexsym, enumerate}
\usepackage[all]{xy}

\numberwithin{equation}{section}

\font\tengothic=eufm10 scaled\magstep 1 \font\sevengothic=eufm7 scaled\magstep 1
\newfam\gothicfam
       \textfont\gothicfam=\tengothic
       \scriptfont\gothicfam=\sevengothic

\newtheorem{theorem}{Theorem}[section]

\newtheorem{corollary}[theorem]{Corollary}

\newtheorem{claim}[theorem]{Claim}
\newtheorem*{claim*}{Claim}

\theoremstyle{definition}
\newtheorem{definition}[theorem]{Definition} 
\newtheorem{remark}[theorem]{Remark}

\newcommand{\length}{\operatorname{length}}

\newcommand {\CCC}{\mathbb{C}}

\newcommand {\PP}{\mathbb{P}}

\begin{document}
\title[Identifiability 
of symmetric tensors]
{A criterion for detecting the identifiability 
of symmetric tensors of size three}

   \author{Edoardo Ballico}
   \address{Edoardo Ballico\\
Dept. of Mathematics\\
 University of Trento\\
38123 Povo (TN), Italy}
\email{ballico@science.unitn.it}

\author{Luca Chiantini}
      \address{Luca Chiantini\\
Universit\'a degli Studi di Siena\\
     Dipartimento di Scienze Matematiche e Informatiche\\
     Pian dei Mantellini, 44\\
     I -- 53100 Siena
     }
   \email{luca.chiantini@unisi.it}

\thanks{The authors were partially supported by MIUR and GNSAGA of INdAM (Italy).}


\date{\today}

\begin{abstract}
We prove a criterion for the identifiability of symmetric
tensors $P$ of type $3\times \dots\times 3$, $d$ times, 
whose rank $k$ is bounded by $(d^2+2d)/8$. The criterion
is based on the study of the Hilbert function of a set of points
$P_1,\dots,  P_k$ which computes the rank of the tensor $P$.
 \end{abstract}



\maketitle


\section{Introduction}

The aim of this paper is to study criteria
which can assure that an explicitly given symmetric tensor, 
whose rank $k$ is known, is identifiable, i.e. 
it can be written {\it uniquely} (up to scalar multiplication
and permutations), as a sum of decomposable tensors.

Recently, new methods for studying
the identifiability of tensors are arising from the
theory of secant varieties to projective varieties,
and their tangential behavior.

In the paper, we deal with symmetric tensors (and their 
geometric counterpart: the space of a Veronese embedding
of a projective space). Let us introduce some
definition, in order to properly state the problem,
along with our achievements.
\smallskip
 
Let $\PP^n:=\PP^n_\CCC$ be a projective space over
the complex field.

Write $\nu_d$ for the $d$-th Veronese map, which sends $\PP^n$
to a space $\PP^N$, with $N=\binom{n+d}n -1$.
The embedding space $\PP^N$ can be seen as the space of
symmetric tensors (up to scalar multiplication) of
type $(n+1)\times (n+1)\times \cdots\times (n+1)$.
We call $n+1$ the {\it size} of these tensors.

The image $X:=\nu_d(\PP^n)$ corresponds to the subset
parameterizing decomposable symmetric tensors (as always: up to
scalar multiplication). The {\it rank} of $P\in\PP^N$
is the minimum $k$ for which there exists an expression
$$ P = P_1 + \dots + P_k$$
with $P_i\in X$ for all $i$.

Identifiability is related with the uniqueness of the
previous expression.

\begin{definition} 
We say that $P\in \PP^N$, of rank $k$, is {\it identifiable} 
if there is a unique expression of $P$ (up to scaling and permutations)
as sum of $k$ elements in $X$.
\end{definition}

In geometric terms, $X$ is a projective variety 
(the {\it $d$-th Veronese variety of $\PP^n$}) and the
rank of a tensor $P$ corresponds to the minimal $k$ such that
$P$ belongs to the {\it standard open subset} $U_k(X)$
of the secant variety $S_k(X)$, formed by $(k-1)$-spaces
spanned by $k$ distinct points of $X$.
\smallskip

Following the classical Terracini's
analysis of the tangent spaces to secant varieties,
one obtains criteria for detecting when a {\it general}
tensor $P$ of rank $k$ is identifiable. An account of
how this can be done can be found in \cite{CC}.
Let we recall briefly what happens for the {\it general}
symmetric tensor of rank $k$.

It is a general non-sense that when the dimension
of the secant variety $U_k(X)$ is not the expected value
(i.e. when $X$ is {\it $(k-1)$-defective}), then also identifiability 
fails. After the results of \cite{AH}, the cases in which 
the Veronese variety $\nu_d(\PP^n)$
is defective, are well known.
On the other hand, there are cases in which the dimension
attains the expected value, and nevertheless the general 
symmetric tensor is not identifiable.
For $n=2$, it is classically known that identifiability
of the generic symmetric tensor fails, besides the 
defective cases, only when  $d=3$ and $k=6$ (see \cite{AC}).
In higher dimension, by the results of \cite{CC}, and the analysis
of the tangential behaviour of Veronese varieties,
carried on by the first author in \cite{B},
one knows that the  general symmetric tensor of rank $k$
is identifiable, as soon as $k< (N+1)/(n+1)$, with
the only possible exception $(n,d)=(3,4)$ (and the defective cases,
listed in \cite{AH}).  

The previous methods, however, only tell
us about generic tensors, but do not apply to
detect whether or not a {\it specific} tensor $P$ is
identifiable.

\begin{remark} Indeed, if we know that the general tensor
of rank $k$ is identifiable, then we can say that
every point, of rank $k$, in the regular locus of $S_k(X)$
is identifiable, by the Zariski Main Theorem.

On the other hand, since the equations for secant
varieties are far from been known, it seems uneasy
to detect directly whether or not a given $P$ belongs
to the singular locus of $S_k(X)$.
\end{remark}

In a private conversation, Joseph Landsberg asked one
of us about the chance of  finding {\it some}
criteria for the identifiability of a specific, given tensor.

Landsberg himself, with Buczy\'{n}ski and Ginensky, found a 
criterion which works for symmetric tensors of any size and
dimension $d$, provided that the rank $k$ is at most
$(d+1)/2$ (see \cite{BGL}). The criterion thus works for
tensors whose rank increases linearly, with respect to $d$. 

Landsberg's problem amounts also to determine methods 
for certifying that a given point of the standard open subset
$U_k(X)\subset S_k(X)$, is not singular.
\smallskip

Following an idea developed by A. Bernardi and the first author
(see \cite{BB}), we are able to produce here,
for the case $n=2$ and in some range for the rank $k$,
a criterion for detecting identifiability.

Our method is based on the study of the Hilbert
function of a set of points
$Z=\{x_1,\dots, x_k\}$, such that
$P =\nu_d(x_1)+\cdots + \nu_d(x_k)$,
i.e. such that $P$ belongs to the linear span
$$ P\in \langle \nu_d(x_1), \dots, \nu_d(x_k)\rangle.$$

Let us recall the following:

\begin{definition}
A set of $k$ distinct points $Z\in\PP^n$ has
{\it general uniform position (GUP)} if for any
$m=1,\dots,k$, no subsets $Z'\subset Z$ of length $m$
belong to hypersurfaces of degree $u$, as soon as
$$ m\geq \binom{u+n}n.$$
 \end{definition}

It is known that general sets of points have GUP,
and the Hilbert function of points with GUP
is well known, when $n=2$, i.e. when $Z$ sits in a plane.

With this in mind, by using standard results 
for the Hilbert function of points in the plane
(we will refer to \cite{EP} for this), as well as by means of 
Lemma 8 in \cite{BB},
we are able to give a criterion for the identifiability of
symmetric tensors of size $3$, i.e. points in
the projective space of the Veronese variety $\nu_d(\PP^2)$.

\begin{theorem}\label{main}
Consider the Veronese variety $X=\nu_d(\PP^2)$, $n>2$, embedded
in the space $\PP^N$, $N=d(d+3)/2$. Let $P\in\PP^N$
be a point of rank $k$, $P=P_1+\cdots + P_k$, with 
$P_i=\nu_d(x_i)$. Assume that the subset $Z=\{x_1,\dots, x_k\}
\subset\PP^2$ has GUP, and
$$ k< \frac {d^2+2d}8.$$
Then $P$ is identifiable.
\end{theorem}  

By means of the Zariski Main Theorem, the previous Theorem 
can be rephrased in terms of the singular locus of $S_k(X)$.

\begin{corollary} 
Let $X$ be the Veronese surface $X=\nu_d(\PP^2)$ and
consider a subset $Z=\{x_1,\dots, x_k\}
\subset\PP^2$ with GUP. Assume  $ k< (d^2+2d)/8$
and consider the span
$$L=\langle \nu_d(x_1),\dots,\nu_d(x_k)\rangle$$
Then $L\cap U_k(X)$ meets the singular locus $Sing(S_k(X))$
only along a subset of $S_{k-1}(X)$.
\end{corollary}

Going back to Landsberg's problem, we notice that
the effectiveness of the criterion for deciding the 
identifiability of a given $P$ depends 
on how much we know about $P$.
In particular, we need to know:
\begin{itemize}
\item the rank $k$ of $P$;
\item one decomposition $P=P_1+\dots +P_k$, $P_i\in X$.
\end{itemize}
Then, assuming that we are in the range 
$ k< (d^2+2d)/8$ (quadratic, with respect to the dimension of
the tensor), it is easy to compute the set
$Z=\{x_1,\dots, x_k\}$, with $P_i=\nu_d(x_i)$, and 
see if it has GUP.

Although these assumptions require a certain
knowledge about the tensor $P$, we hope that 
the criterion could be effective, in some 
interesting cases.

On the other hand, the criterion has some intriguing geometric
aspects. To mention one: a link between the postulation
of $Z$ and the identifiability of points in $\langle Z\rangle$.

It is, in any event, a starting point. We cannot exclude that,
on the same lines, it will be possible to  produce criteria
with a wider range of applicability.

\section{Proof of the criterion}

We keep here the notation of the Introduction, from the geometric
point of view.

So, $X=\nu_d(\PP^2)$ is the $d$-th Veronese embedding of the plane
in $\PP^N$, $N=(d+3)d/2$. $P$ is a point of $\PP^N$, which 
has rank $k>1$. Fix $k$ points $x_1,\dots,
x_k$ of $\PP^2$ such that 
$$P\in\langle \nu_d(x_1),\dots, \nu_d(x_k)\rangle.$$

Write $Z=\{x_1,\dots,x_k\}$. We make the following
assumptions:
\begin{itemize}
\item $ k< (d^2+2d)/8$; 
\item $Z$ has GUP.
\end{itemize}

We want to prove that $P$ is identifiable.
\smallskip

Assume, on the contrary, that there is another subscheme
$Z'\subset \PP^2$, $Z'=\{y_1,\dots,y_k\}$, of length $k$, 
such that $P\in\langle \nu_d(Z')
\rangle$.

Call $W$ the union $W=Z\cup Z'$, which is a subscheme
of length  $w\leq 2k$. We will look carefully at the
Hilbert function $h_W$ of $W$ and at its difference function
$Dh_W$.

\begin{claim} $h_W(d)<w$, so that $Dh_W(d+1)>0$.
\end{claim}
\begin{proof}
By our first assumption on $P$, it turns out that 
the linear spans of both $\nu_d(Z)$ and $\nu_d(Z')$
have dimension $k$. Moreover they meet in a point
$P$ which cannot lie in the linear span of the intersection
$\nu_d(Z)\cap\nu_d(Z')$. It follows that
$\nu_k(W)$ does not impose $w=2k-\length(\nu_d(Z)\cap\nu_d(Z'))$
conditions to the hyperplanes of $\PP^N$, from which
the claims on $h_W$ follows at once.
\end{proof}

Now, we use the numerical assumption, together
with a knowledge of the main properties of
Hilbert functions of subsets of $\PP^2$.

\begin{claim}\label{conti}
Let $u$ be the integer such that:
$$ \frac{u^2+3u+2}2\leq k<\frac{(u+1)^2+3(u+1)+2}2.$$
Then $u+2\leq d/2$ and the function $Dh_W$ satisfies:
$$ Dh_W(i) =i+1 \mbox{ for } i=0,\dots,u.$$
\end{claim}
\begin{proof} The first inequality follows immediately
from the assumption $ k< (d^2+2d)/8$. The second one
follows from the fact that $Z$ has GUP, and thus
$$ h^0(I_W(u))\leq h^0(I_Z(u))=0.$$
\end{proof}

\begin{claim}\label{pianoro}
There exists some $j\leq d$ with
$$u+1> Dh_W(j) = Dh_W(j+1)>0.$$
\end{claim}
\begin{proof}
First observe that the definition of $u$ and the
numerical assumption $ k< (d^2+2d)/8$ imply that:
\begin{equation}\label{numer}
2k\leq (u+1)d-u^2+u+2.
\end{equation}
On the other hand, if the quoted $j$ does not exist,
we must have
$$Dh_W(d+1-i) \geq i+1 \mbox{ for } i=0,\dots,u
$$
which gives, after a short computation,
$$\length(W)\geq (u+1)d-u^2+u+2,$$
a contradiction.
\end{proof}

Define the number $m$ as:
$$m=\min\{Dh_W(j): j\leq d and Dh_W(j) = Dh_W(j+1)>0\}.$$

By the previous claim, we know that $m\leq u$.

\begin{claim}
There exists a curve $M\subset \PP^2$, of degree $m$,
such that $M$ intersects $W$ in a subset $A$ of length
$a\leq (m+1)d-m^2+m+2.$
\end{claim}
\begin{proof}
This is an easy consequence of well known facts
on sets of points $W$ in the plane, whose function
$Dh_W$ has the behavior described in Claim \ref{pianoro}.

Namely (see e.g. \cite{D} or \cite{EP}, Proposition at p.112), 
since  
$$\max\{Dh_W(i)\}>m=  Dh_W(j) = Dh_W(j+1)>0,$$
we know that there exists a curve $M$ of degree $m$
which meets $W$ in a subset $A$ whose Hilbert function
is defined as:
$$ Dh_A= \min\{m,Dh_W\}.$$
It follows immediately that the length of $A$ is at least:
$$ a = \frac{m(m+1)}2 + m(d+2-2m) + \frac{m(m+1)}2, $$
which gives the claim.
\end{proof}

Now we have all the ingredients for the:
\smallskip

{\bf Proof of the main Theorem.} 
Define $B=W\setminus M$. By  \cite{EP} p.112, we know that
the function $Dh_B$ satisfies, for all $i$:
 $Dh_A(i)+Dh_B(i-m)=Dh_W(i)$.

Since $Dh_W(j)\leq m$, then  $ Dh_A(j)= \min\{m,Dh_W(j)\}=
Dh_W(j)$, so that one has $Dh_B(d-m)=0$. It follows from
\cite{BB}, Lemma 8, that  
$Z-M=Z'-M$, which implies that $Z\cap M$ has the same cardinality
as $Z'\cap M$.

Thus  $Z\cap M$ has cardinality:
$$\frac{\length(A)}2\geq \frac{(m+1)d-m^2+m+2}2,$$
and sits in a curve of degree $M$.
Since $Z$ has GUP, we get that:
$$ \frac{(m+1)d-m^2+m+2}2<\frac{m^2+3m+2}2.$$
This implies $d\leq 2m$,
which is impossible, since $m<u$ and $u< d/2$,
by Claim \ref{conti}.
\qed\smallskip

\begin{remark}
Let us notice that, with the same method, one
can prove a slightly stronger condition. Namely,
with the previous assumptions, it follows that
$P$ cannot belong to the linear span of another subset
$\nu_d(Z')$, with $\deg(Z')=k$, dropping the
assumption that $Z'$ is reduced.

Indeed, in this case, we may define $W=Z'\cup (Z\setminus Z')$.
The results on the Hilbert function of $0$-dimensional subsets 
of $\PP^2$, obtained in \cite{EP}, remain true even if
$W$ is not reduced, as well as Lemma 8 of \cite{BB}.
So, the previous arguments work verbatim.
\end{remark}


\begin{thebibliography}{\index}

\bibitem[AH95]{AH} {\sc J. Alexander and A. Hirschowitz}. 
\emph{Polynomial interpolation in several variables.} J. Alg. Geom.
\textbf{4} (1995), 201-222.

\bibitem[AC81]{AC} {\sc E. Arbarello and M. Cornalba}. 
\emph{Footnotes to a paper of B. Segre}. Math. Ann. \textbf{256} (1981),
341-362.

\bibitem[B06]{B} {\sc E. Ballico.} \emph{On the non-defectivity
and weak non-defectivity of Segre-Veronese embeddings
of products of projective spaces.}
Portug. Mat. \textbf{63} (2006), 101--111.

\bibitem[BB10]{BB} {\sc E. Ballico and A. Bernardi.} 
\emph{A partial stratification 
of secant varieties of Veronese varieties via curvilinear subschemes.} 
preprint arXiv:1010.3546v2 (2010).

\bibitem[BGL10]{BGL} {\sc J. Buczy\'{n}ski, A. Ginensky and J. M. Landsberg.}
\emph{Determinantal equations for secant varieties and the 
Eisenbud-Koh-Stillman conjecture.} preprint arXiv:1007.0192v2 (2010).

\bibitem[CC06]{CC} {\sc L. Chiantini and C. Ciliberto.} 
\emph{On the concept of $k$--secant 
order of a variety.}
 J. London Math. Soc. \textbf{73} (2006), 436--454.
 
\bibitem[D86]{D} {\sc E. Davis.} 
\emph{$0$-dimensional subschemes of $\PP^2$.}
 Queens Papers in Pure and Appl. Math. \textbf{76} (1986). 

\bibitem[EP90]{EP} {\sc P. Ellia and C. Peskine.} Groupes de points de 
${\bf {P}}^2$: caract\`{e}re et position uniforme.
\emph{Algebraic geometry (L'Aquila, 1988). Springer
Lecture Notes in Math.} \textbf{1417}  (1990),  111--116.

\end{thebibliography}
\end{document}